\documentclass[hidelinks,onefignum,onetabnum]{siamart220329}

\ifpdf
\hypersetup{
  pdftitle={Convergence of the Dirichlet--Neumann method for semilinear elliptic equations},
  pdfauthor={Emil Engström}
}
\fi

\usepackage{amsmath}
\usepackage{amssymb}
\usepackage{subcaption}

\usepackage{xcolor}
\usepackage{hyperref}
\hypersetup{colorlinks=true, linkcolor=blue, citecolor=blue}

\newsiamremark{assumption}{Assumption}
\crefname{assumption}{Assumption}{Assumptions}
\newsiamremark{remark}{Remark}
\crefname{remark}{Remark}{Remarks}
\newsiamremark{example}{Example}
\crefname{example}{Example}{Examples}
\crefname{section}{Section}{Sections}

\newcommand{\R}{\mathbb{R}}

\newcommand\restr[2]{\ensuremath{\left.#1\right\vert_{#2}}}

\begin{document}
    \title{Convergence of the Dirichlet--Neumann method for semilinear elliptic equations\thanks{Version of  \today.
\funding{This work was supported by the Swedish Research Council under the grant  2023--04862}}}
    \author{Emil Engstr\"om\thanks{Centre for Mathematical Sciences, Lund University,
P.O.\ Box 118, SE-22100 Lund, Sweden, (\email{emil.engstrom@math.lth.se}).}}
	\maketitle
 \begin{abstract}
     The Dirichlet--Neumann method is a common domain decomposition method for nonoverlapping domain decomposition and the method has been studied extensively for linear elliptic equations. However, for nonlinear elliptic equations, there are only convergence results for some specific cases in one spatial dimension. The aim of this manuscript is therefore to prove that the Dirichlet--Neumann method converges for a class of semilinear elliptic equations on Lipschitz continuous domains in two and three spatial dimensions. This is achieved by first proving a new result on the convergence of nonlinear iterations in Hilbert spaces and then applying this result to the Steklov--Poincaré formulation of the Dirichlet--Neumann method.
 \end{abstract}



 \section{Introduction}
 We consider the semilinear elliptic equation with homogeneous Dirichlet boundary conditions
\begin{equation}\label{eq:strong}
    \left\{
         \begin{aligned}
                -\nabla\cdot\bigl(\alpha(x)\nabla u\bigr)+\beta(x, u)&=f  & &\text{in }\Omega,\\
                 u&=0 & &\text{on }\partial\Omega.
            \end{aligned}
    \right.
\end{equation}
Here, $\Omega\subset\R^d$ with $d=2, 3$ is a Lipschitz domain and $\partial\Omega$ denotes the boundary of $\Omega$. In order to solve~\cref{eq:strong} in parallel it is a common practice to employ a nonoverlapping domain decomposition method. For linear elliptic equations, the convergence theory of such methods is well established~\cite{quarteroni,widlund}. For overlapping Schwarz methods applied to nonlinear elliptic equations there is also a general theory, see, e.g.~\cite{dryja97,tai98,tai02}. However, concerning nonoverlapping domain decompositions for nonlinear elliptic equations there are few results. Some exceptions are~\cite{berninger11} that show convergence of the Dirichlet--Neumann and Robin--Robin method for a class of nonlinear elliptic problems and~\cite{gander22}, in which quadratic convergence is proven for specific parameter choices of the Dirichlet--Neumann method. Both of these results are only proven on one-dimensional domains, however.

For convergence on general Lipschitz domains in $\R^d$, $d=2, 3$ there are the results~\cite{eeeh22,eeeh24apnum} concerning the Robin--Robin method and modified Neumann--Neumann methods, respectively. However, the convergence rate of the Robin--Robin method deteriorates as the mesh size $h$ is made smaller~\cite{lui09}. Moreover, the Neumann--Neumann method does not seem to converge at all for some nonlinear problems~\cite{eeeh24apnum} and, while the modified Neumann--Neumann methods perform better in this regard, their convergence depend on the choice of an auxiliary problem, which may not be obvious how to choose. The Dirichlet--Neumann method, on the other hand, is expected to have mesh independent convergence and does not require an auxiliary problem. Moreover, our numerical results indicate that the method converges faster than the Robin--Robin method and converges in some cases when the Neumann--Neumann method does not converge at all. To our knowledge there are currently no general convergence proofs for the Dirichlet--Neumann method as there is for the Robin--Robin and modified Neumann-Neumann methods. The aim of this manuscript is therefore to prove that the Dirichlet--Neumann method converges for a large class of semilinear elliptic equations without restrictive regularity assumptions on the equation or the domain.

We decompose the domain $\Omega$ into the subdomains $\Omega_1$ and $\Omega_2$ such that
\begin{displaymath}
\overline{\Omega}=\overline{\Omega}_1\cup\overline{\Omega}_2,\quad \Omega_1\cap\Omega_2=\emptyset\quad\text{and}\quad\Gamma=(\partial\Omega_1\cap\partial\Omega_2)\setminus\partial\Omega,
\end{displaymath}
where we refer to $\Gamma$ as the interface of the decomposition. See~\cref{fig:dom1} for an example of such a decomposition. It is also possible to consider two families of nonadjacent subdomains as in~\cref{fig:dom2}, which will yield a method that can be implemented in parallel. We do not consider decompositions with cross-points, but these have been studied for the Dirichlet--Neumann method applied to linear equations in~\cite{dumas}.
\begin{figure}
\centering
\begin{subfigure}{.4\textwidth}
  \centering
  \includegraphics[width=1.0\linewidth]{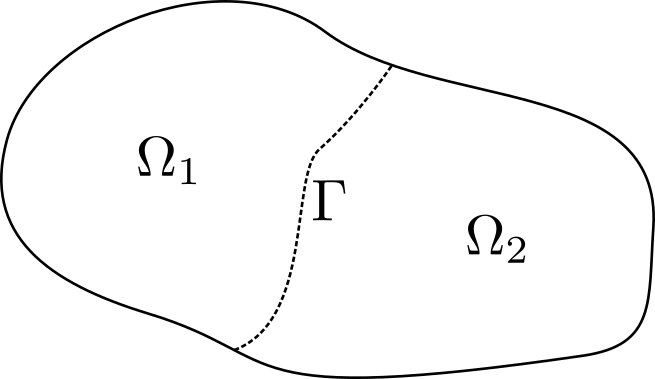}
  \caption{A decomposition with two subdomains $\Omega_1$ and $\Omega_2$}
  \label{fig:dom1}
\end{subfigure}%
\hspace{1em}
\begin{subfigure}{.4\textwidth}
  \centering
  \includegraphics[width=1.0\linewidth]{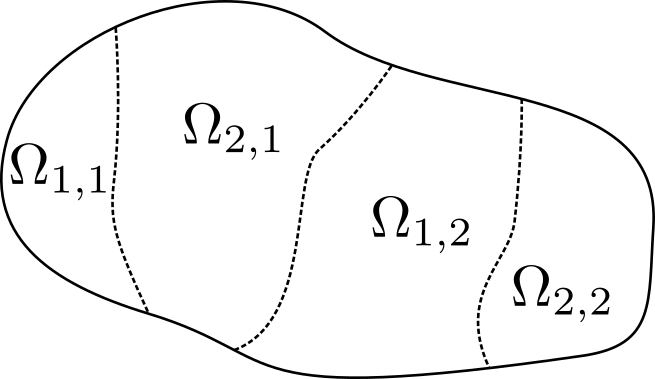}
  \caption{A decomposition into two families of subdomains $\{\Omega_{1,\ell}\}$ and $\{\Omega_{2,\ell}\}$.}
  \label{fig:dom2}
\end{subfigure}
\end{figure}

The Dirichlet--Neumann method applied to the semilinear equation~\cref{eq:strong} is the following. Let $f_i=\restr{f}{\Omega_i}$ for $i=1, 2$. For a method parameter $s>0$ and an initial guess $\eta^0$ on $\Gamma$, we find $(u_1^{n+1}, u_2^{n+1})$ for $n=0,1,2,\ldots,$  such that
\begin{equation}\label{eq:dn}
\left\{
\begin{aligned}
-\nabla\cdot(\alpha(x)\nabla u^{n+1}_1)+\beta(x, u^{n+1}_1)&=f_1  & &\text{in }\Omega_1,\\
u^{n+1}_1&=0 & &\text{on }\partial\Omega_1\setminus\Gamma,\\
u^{n+1}_1 &= \eta^n & &\text{on }\Gamma,\\[5pt]
-\nabla\cdot(\alpha(x)\nabla  u^{n+1}_2)+\beta(x, u^{n+1}_2)&=f_2  & &\text{in }\Omega_2,\\
u^{n+1}_2&=0 & &\text{on }\partial\Omega_2\setminus\Gamma,\\
\alpha(x)\nabla u^{n+1}_2\cdot\nu_2 &=  \alpha(x)\nabla u^{n+1}_1\cdot\nu_2 & &\text{on }\Gamma,
\end{aligned}
\right.
\end{equation}
with $\eta^{n+1}=s\restr{u_2^{n+1}}{\Gamma}+(1-s)\eta^n$. Here, $\nu_2$ denotes the outwards pointing unit normal of $\Omega_2$. As will be shown, the Dirichlet--Neumann method corresponds to the interface iteration
\begin{equation}\label{eq:dn_intro}
    \eta^{n+1} = (1-s)\eta^n+sS_2^{-1}(0-S_1\eta^n),
\end{equation}
where $S_i$ for $i=1, 2$ are the Steklov--Poincaré operators.

The paper is organized in the following way. In~\cref{sec:abstract} we present a new convergence result for iterations of the form~\cref{eq:dn_intro} in Hilbert spaces. We then introduce the weak formulations of the nonlinear operators in~\cref{eq:dn} and show some important properties of these in~\cref{sec:weak}. In~\cref{sec:ddm} we introduce the domain decomposition, transmission problem and Steklov--Poincaré operators. We then show that our nonlinear Steklov--Poincaré operators satisfy the conditions of our abstract theorem. In~\cref{sec:dn} we introduce the weak formulation of the Dirichlet--Neumann method and show that it can be interpreted as an interface iteration. We then conclude that the Dirichlet--Neumann method converges. Finally, we give some numerical results that verify our convergence theorem in~\cref{sec:num} and compare the convergence rate with the Robin--Robin and Neumann--Neumann methods. We also give some numerical results for an equation not satisfying the conditions to demonstrate the generality of the nonlinear Dirichlet--Neumann method. Throughout the manuscript we will denote arbitrary positive constants that may change between lines by $c, C>0$.
\section{Abstract theorem on nonlinear splitting in Hilbert spaces}\label{sec:abstract}
    Let $X$ and $Y$ be two Hilbert spaces. We denote the space of bounded linear operators from $X$ to $Y$ by $B(X, Y)$. We say that a (nonlinear) operator $G:X\rightarrow Y$ is bounded if for any bounded set $U$, the set $G(U)$ is bounded. Moreover, $G$ is Lipschitz continuous on a subset $U\subset X$ if there exists a constant $C>0$ such that
     \begin{displaymath}
         \|G\mu-G\lambda\|_Y\leq C\|\mu-\lambda\|_X\qquad\text{for all }\mu, \lambda\in U.
     \end{displaymath}
 The operator $G$ is Fréchet differentiable on the subset $U\subset X$ if there exists an operator, the Fréchet derivative, $G':U\rightarrow B(X, Y)$ such that, for any $\mu\in U$, we have
 \begin{displaymath}
     \lim_{h\rightarrow 0}\frac{\|G(\mu+h)-G(\mu)-G'(\mu)h\|_Y}{\|h\|_X}=0.
 \end{displaymath}
 If $G$ is Lipschitz continuous or Fréchet differentiable on $U=X$, we simply say that $G$ is Lipschitz continuous or Fréchet differentiable, respectively. We now recall the Fundamental theorem of calculus. For a reference, see the proof of~\cite[Lemma 9.12]{teschl}. If $G$ is Fréchet differentiable on a convex set $U$ with continuous Fréchet derivative $G'$ we have, for $\mu, \lambda\in U$,
\begin{equation}\label{eq:ftc}
    G\mu-G\lambda=\int_0^1 G'\bigl(\lambda+t(\mu-\lambda)\bigr)(\mu-\lambda)\mathrm{d}t.
\end{equation}
Here, and for the rest of the manuscript, the integral should be interpreted as the Bochner integral, see~\cite[Chapter 9]{teschl}.

For a Hilbert space $X$ we denote its dual by $X^*$ and the dual pairing by $\langle \cdot, \cdot\rangle_{X^*\times X}$. If the spaces are obvious from the context we denote the dual paring by $\langle\cdot, \cdot\rangle$. The Bochner integral on $X^*$ satisfies
\begin{equation}\label{eq:intinner}
    \Bigl\langle\int_0^1\mu(t)\mathrm{d}t, \lambda\Bigr\rangle=\int_0^1 \langle \mu(t), \lambda\rangle\mathrm{d}t\qquad\text{for all } \mu\in C\bigl([0, 1], X^*\bigr), \lambda\in X,
\end{equation}
see~\cite[(9.7)]{teschl}. We say that a linear operator $P:X\rightarrow X^*$ is symmetric if
    \begin{displaymath}
     \langle P\mu, \lambda\rangle=\langle P\lambda, \mu\rangle\qquad\text{for all }\mu, \lambda\in X
    \end{displaymath}
    and denote the subspace of symmetric and bounded linear operators by $S(X, X^*)\subset B(X, X^*)$. Moreover, we say that $P$ is coercive if
     \begin{displaymath}
         \langle P\mu, \mu\rangle\geq c\|\mu\|_X^2\qquad\text{for all }\mu\in X.
     \end{displaymath}

 For $G:X\rightarrow X^*$ Lipschitz continuity on $U$ is equivalent to
     \begin{displaymath}
         |\langle G\mu-G\lambda, \nu\rangle|\leq C\|\mu-\lambda\|_X\|\nu\|_X\qquad\text{for all }\mu, \lambda\in U, \nu\in X.
     \end{displaymath}
 Furthermore, we say that $G:X\rightarrow X^*$ is uniformly monotone if there exists a constant $c>0$ such that
 \begin{displaymath}
     \langle G\mu-G\lambda, \mu-\lambda\rangle\geq c\|\mu-\lambda\|_X^2\qquad\text{for all }\mu, \lambda\in X.
 \end{displaymath}

 We now consider a Hilbert space $X$, a nonlinear operator $G:X\rightarrow X^*$, and $\chi\in X^*$. The abstract problem is to find $\eta\in X$ such that
 \begin{equation}\label{eq:abstracteq}
     G\eta=\chi.
 \end{equation}
 A nonlinear splitting can be defined in the following way. Suppose that
 \begin{displaymath}
    G=G_1+G_2,
\end{displaymath}
where $G_i:X\rightarrow X^*$, $i=1, 2$. For some method parameter $s>0$ and initial guess $\eta^0$ the iterative method is
\begin{equation}\label{eq:abstractit}
    \eta^{n+1} = (1-s)\eta^n+sG_2^{-1}(\chi-G_1\eta^n).
\end{equation}
The main abstract result of this section is the convergence of this iteration.
 \begin{theorem}\label{thm:abstract}
    Let $X$ be a (real) Hilbert space, $G:X \rightarrow X^*$ be a uniformly monotone operator, $\chi\in X^*$, and $\eta^0\in X$. 
    \begin{itemize}
        \item Suppose that for any bounded subset $\Tilde{U}\subset X$ the operator $G$ is Lipschitz continuous on $\Tilde{U}$. Then $G:X\rightarrow X^*$ is bijective and, in particular, there exists a unique solution to~\cref{eq:abstracteq}. Moreover, the inverse $G^{-1}:X^*\rightarrow X$ is Lipschitz continuous.
        \item Let $G=G_1+G_2$, where $G_1, G_2$ are uniformly monotone. Suppose that for any bounded subset $\Tilde{U}\subset X$ the nonlinear operator $G_2$ is Lipschitz continuous on $\Tilde{U}$. Then $G_2:X\rightarrow X^*$ is bijective and the iteration~\cref{eq:abstractit} is well defined. Moreover, the inverse $G_2^{-1}:X^*\rightarrow X$ is Lipschitz continuous.
        \item In addition to the above, let $U\subset X$ be a bounded, open, and convex subset such that the solution $\eta$ of~\cref{eq:abstracteq} satisfies $\eta\in U$. Suppose that $G_1$ is Lipschitz continuous on $U$. Moreover, suppose that $G_2$ is Fréchet differentiable on $U$ with Lipschitz continuous Fréchet derivative $G_2': U\rightarrow S(X, X^*)$. Then, for $\eta^0$ close enough to $\eta$ and sufficiently small $s>0$, the iteration~\cref{eq:abstractit} converges to $\eta$. The convergence is linear, i.e., there exists constants $C>0$ and $0\leq L < 1$ such that
    \begin{displaymath}
    	   \|\eta^n-\eta\|_X\leq C L^n\|\eta^0-\eta\|_X.
    \end{displaymath}
    \end{itemize}
\end{theorem}
\begin{proof}
    For the first two statements we use the Browder--Minty theorem~\cite[Theorem 26.A]{zeidler}. We show the required properties only of $G$ since the case of $G_2$ follows in the same way. The first property is coercivity, i.e., if $\|\mu\|_X\rightarrow \infty$ then 
    \begin{displaymath}
        \frac{\langle G\mu, \mu\rangle}{\|\mu\|_X}\rightarrow\infty.
    \end{displaymath}
    This follows immediately from the uniform monotonicity of $G$ since
    \begin{displaymath}
        \langle G\mu, \mu\rangle=\langle G\mu-G(0), \mu-0\rangle+\langle G(0), \mu\rangle\geq c\|\mu\|_X^2-\|G(0)\|_{X^*}\|\mu\|_X.
    \end{displaymath}
    The fact that $G$ is monotone, i.e.,
    \begin{displaymath}
        \langle G\mu-G\lambda, \mu-\lambda\rangle \geq 0,
    \end{displaymath}
    is obvious from the uniform monotonicity. Finally, we show that the operator $G$ is demicontinuous, i.e., that
        \begin{displaymath}
            \langle G\mu^n-G\mu, \lambda\rangle\rightarrow 0
    \end{displaymath}
     as $\mu^n\rightarrow\mu$ in $X$ for all $\lambda\in X$. To see this, let $\mu^n\rightarrow\mu$. Then $\|\mu^n\|\leq C$ and therefore there exists a bounded set $\Tilde{U}$ such that $\{\mu^n\}\cup\{\mu\}\subset \Tilde{U}$. Since $G$ is Lipschitz continuous on $\Tilde{U}$, we have that
        \begin{displaymath}
            \langle G\mu^n-G\mu, \lambda\rangle\leq C(\Tilde{U})\|\mu^n-\mu\|_X\|\lambda\|_X\rightarrow 0
    \end{displaymath}
    for all $\lambda\in X$. We can now apply the Browder--Minty theorem to conclude that $G$ is bijective. To see that the inverse is Lipschitz continuous, we consider $\mu^*,\lambda^*\in X^*$ and apply the uniform monotonicity of $G$ as follows
    \begin{align*}
        \|G^{-1}\mu^*-G^{-1}\lambda^*\|_X^2&\leq \langle GG^{-1}\mu^*-GG^{-1}\lambda^*, G^{-1}\mu^*-G^{-1}\lambda^*\rangle\\
        &\leq\| \mu^*-\lambda^*\|_{X^*}\|G^{-1}\mu^*-G^{-1}\lambda^*\|_X.
    \end{align*}
    Dividing by $\|G^{-1}\mu^*-G^{-1}\lambda^*\|_X$ shows that $G^{-1}$ is Lipschitz continuous on $X^*$.
    
 To show the final statement, let $U$ be a bounded, open, and convex subset such that $\eta\in U$ and consider the operator $K:X\to X$ defined by
 \begin{displaymath}
     K\mu=(1-s)\mu+s G_2^{-1}(\chi-G_1\mu).
 \end{displaymath}
 Since $G_2^{-1}$ is Lipschitz continuous and $G_1$ is Lipschitz continuous on $U$ we have that $G_2^{-1}(\chi-G_1\mu)$ and $K$ also are Lipschitz continuous on $U$. Moreover, since $\eta$ is a fixed point of both $K$ and $G_2^{-1}(\chi-G_1\mu)$ we can find a convex open set $W\subset U$ such that $\eta\in W$ and for $\mu\in W$ we have $K\mu\in U$ and $G_2^{-1}(\chi-G_1\mu)\in U$.

We now fix arbitrary $\mu,\lambda\in W$ and for $t\in[0, 1]$ we define the lines
\begin{align*}
    x(t)&=K\lambda+t(K\mu-K\lambda),\\
    y(t)&=G_2^{-1}(\chi-G_1\lambda) + t(G_2^{-1}(\chi-G_1\mu)-G_2^{-1}(\chi-G_1\lambda)),\quad\text{and}\\
    z(t)&=\lambda+t(\mu-\lambda).
\end{align*}
Since $U$ is convex, we have $x, y, z\subset U$. Thus, by~\cref{eq:ftc} we have
\begin{align}
    G_2K\mu-G_2K\lambda&=\int_0^1 G_2'\bigl(x(t)\bigr)(K\mu-K\lambda)\mathrm{d}t\label{eq:G2K}\\
    G_1\lambda-G_1\mu&=\int_0^1 G_2'\bigl(y(t)\bigr)(G_2^{-1}(\chi-G_1\mu)-G_2^{-1}(\chi-G_1\lambda))\mathrm{d}t\label{eq:G1}\\
    G_2\mu-G_2\lambda&=\int_0^1 G_2'\bigl(z(t)\bigr)(\mu-\lambda)\mathrm{d}t.\label{eq:G2}
\end{align}
Note that for the second identity we have used
\begin{displaymath}
G_1\lambda-G_1\mu=G_2G_2^{-1}(\chi-G_1\mu)-G_2G_2^{-1}(\chi-G_1\lambda).
\end{displaymath}
Using~\cref{eq:G2K,eq:intinner}, we get
\begin{align*}
&\langle G_2K\mu-G_2K\lambda, K\mu-K\lambda\rangle=\biggl\langle\int_0^1 G_2'\bigl(x(t)\bigr)(K\mu-K\lambda)\mathrm{d}t, K\mu-K\lambda\biggr\rangle\\
&\quad=\int_0^1\Bigl\langle G_2'\bigl(x(t)\bigr)(K\mu-K\lambda), K\mu-K\lambda\Bigr\rangle\mathrm{d}t\\
&\quad=\int_0^1\Bigl\langle G_2'\bigl(x(t)\bigr)\bigl((1-s)(\mu-\lambda)\\
&\qquad+sG_2^{-1}(\chi-G_1\mu)-sG_2^{-1}(\chi-G_1\lambda)\bigr), K\mu-K\lambda\Bigr\rangle\mathrm{d}t.
\end{align*}
Expanding $K\mu-K\lambda$ again, we find
\begin{align*}
&\langle G_2K\mu-G_2K\lambda, K\mu-K\lambda\rangle=(1-s)\int_0^1\Bigl\langle G_2'\bigl(x(t)\bigr)(\mu-\lambda), K\mu-K\lambda\Bigr\rangle\mathrm{d}t\\
&\qquad+s\int_0^1\Bigl\langle G_2'\bigl(x(t)\bigr)\bigl(G_2^{-1}(\chi-G_1\mu)-G_2^{-1}(\chi-G_1\lambda)\bigr), K\mu-K\lambda\Bigr\rangle\mathrm{d}t\\
&\quad=(1-s)^2\int_0^1\Bigl\langle G_2'\bigl(x(t)\bigr)(\mu-\lambda), \mu-\lambda\Bigr\rangle\mathrm{d}t\\
&\qquad+(1-s)s\int_0^1\Bigl\langle G_2'\bigl(x(t)\bigr)(\mu-\lambda), G_2^{-1}(\chi-G_1\mu)-G_2^{-1}(\chi-G_1\lambda)\Bigr\rangle\mathrm{d}t\\
&\qquad+s(1-s)\int_0^1\Bigl\langle G_2'\bigl(x(t)\bigr)\bigl(G_2^{-1}(\chi-G_1\mu)-G_2^{-1}(\chi-G_1\lambda)\bigr), \mu-\lambda\Bigr\rangle\mathrm{d}t\\
&\qquad+s^2\int_0^1\Bigl\langle G_2'\bigl(x(t)\bigr)\bigl(G_2^{-1}(\chi-G_1\mu)-G_2^{-1}(\chi-G_1\lambda)\bigr), \\
&\qquad\qquad\qquad G_2^{-1}(\chi-G_1\mu)-G_2^{-1}(\chi-G_1\lambda)\Bigr\rangle\mathrm{d}t.
\end{align*}
Using the symmetry of $G_2'(x(t))$ now yields
\begin{align*}
&\langle G_2K\mu-G_2K\lambda, K\mu-K\lambda\rangle=(1-s)^2\int_0^1\Bigl\langle G_2'\bigl(x(t)\bigr)(\mu-\lambda), \mu-\lambda\Bigr\rangle\mathrm{d}t\\
&\qquad+2s(1-s)\int_0^1\Bigl\langle G_2'\bigl(x(t)\bigr)\bigl(G_2^{-1}(\chi-G_1\mu)-G_2^{-1}(\chi-G_1\lambda)\bigr), \mu-\lambda\Bigr\rangle\mathrm{d}t\\
&\qquad+s^2\int_0^1\Bigl\langle G_2'\bigl(x(t)\bigr)\bigl(G_2^{-1}(\chi-G_1\mu)-G_2^{-1}(\chi-G_1\lambda)\bigr),\\
&\qquad\qquad\qquad G_2^{-1}(\chi-G_1\mu)-G_2^{-1}(\chi-G_1\lambda)\Bigl\rangle\mathrm{d}t\\
&\quad=(1-s)^2I_1+2s(1-s)I_2+s^2I_3.
\end{align*}
By~\cref{eq:G2}, the first term is
\begin{align*}
I_1&=\int_0^1\Bigl\langle G_2'\bigl(x(t)\bigr)(\mu-\lambda), \mu-\lambda\Bigr\rangle\mathrm{d}t\\
&=\int_0^1\Bigl\langle G_2'\bigl(z(t)\bigr)(\mu-\lambda), \mu-\lambda\Bigr\rangle\mathrm{d}t+\int_0^1\Bigl\langle \Bigl(G_2'\bigl(x(t)\bigr)-G_2'\bigl(z(t)\bigr)\Bigr)(\mu-\lambda), \mu-\lambda\Bigr\rangle\mathrm{d}t\\
&=\langle G_2\mu-G_2\lambda, \mu-\lambda\rangle+\int_0^1\Bigl\langle \Bigl(G_2'\bigl(x(t)\bigr)-G_2'\bigl(z(t)\bigr)\Bigr)(\mu-\lambda), \mu-\lambda\Bigr\rangle\mathrm{d}t
\end{align*}
and thus by the Lipschitz continuity of $G_2'$ on $U$ we have
\begin{align*}
I_1&\leq\langle G_2\mu-G_2\lambda, \mu-\lambda\rangle+C\|\mu-\lambda\|_X^2\int_0^1\|x(t)-z(t)\|_X\mathrm{d}t.
\end{align*}
Similarly, using~\cref{eq:G1}, the second term can be written
\begin{align*}
I_2&=\int_0^1\Bigl\langle G_2'\bigl(y(t)\bigr)\bigl(G_2^{-1}(\chi-G_1\mu)-G_2^{-1}(\chi-G_1\lambda)\bigr), \mu-\lambda\Bigr\rangle\mathrm{d}t\\
&\quad+\int_0^1\Bigl\langle \Bigl(G_2'\bigl(x(t)\bigr)-G_2'\bigl(y(t)\bigr)\Bigr)\bigl(G_2^{-1}(\chi-G_1\mu)-G_2^{-1}(\chi-G_1\lambda)\bigr), \mu-\lambda\Bigr\rangle\mathrm{d}t\\
&=-\langle G_1\mu-G_1\lambda, \mu-\lambda\rangle\\
&\quad+\int_0^1\Bigl\langle \bigl(G_2'\bigl(x(t)\bigr)-G_2'\bigl(y(t)\bigr)\bigr)\bigl(G_2^{-1}(\chi-G_1\mu)-G_2^{-1}(\chi-G_1\lambda)\bigr), \mu-\lambda\Bigr\rangle\mathrm{d}t.
\end{align*}
By the Lipschitz continuity of $G_1$ and $G_2'$ on $U$ and the Lipschitz continuity of $G_2^{-1}$ on $X^*$, this gives the estimate
\begin{align*}
I_2&\leq-\langle G_1\mu-G_1\lambda, \mu-\lambda\rangle+C\|\mu-\lambda\|_X^2\int_0^1\|x(t)-y(t)\|_X\mathrm{d}t.
\end{align*}
For the same reasons as above and the fact that $G_2'$ being Lipschitz on $U$ implies that it is bounded on $U$, we also have the estimate
\begin{align*}
I_3&\leq C\|\mu-\lambda\|_X^2.
\end{align*}
Putting this together gives
\begin{align*}
&\langle G_2K\mu-G_2K\lambda, K\mu-K\lambda\rangle\\
&\quad\leq (1-s)^2\langle G_2\mu-G_2\lambda, \mu-\lambda\rangle-2s(1-s)\langle G_1\mu-G_1\lambda, \mu-\lambda\rangle+s^2C\|\mu-\lambda\|_X^2\\
&\qquad+ C\|\mu-\lambda\|_X^2\int_0^1\bigl(\|x(t)-y(t)\|_X+\|x(t)-z(t)\|_X\bigr)\mathrm{d}t.
\end{align*}
Using the uniform monotonicity of $G_1$ and the Lipschitz continuity of $G_2$ on $U$ yields
\begin{align*}
    -\langle G_1\mu-G_1\lambda, \mu-\lambda\rangle\leq -c\|\mu-\lambda\|_X^2\leq -C\langle G_2\mu-G_2\lambda, \mu-\lambda\rangle.
\end{align*}
This, together with the uniform monotonicity of $G_2$, gives
\begin{align*}
&\langle G_2K\mu-G_2K\lambda, K\mu-K\lambda\rangle\\
&\quad\leq (1-s)^2\langle G_2\mu-G_2\lambda, \mu-\lambda\rangle-2s(1-s)c\langle G_2\mu-G_2\lambda, \mu-\lambda\rangle\\
&\qquad+s^2C\langle G_2\mu-G_2\lambda, \mu-\lambda\rangle\\
&\qquad+ C\langle G_2\mu-G_2\lambda, \mu-\lambda\rangle\int_0^1\bigl(\|x(t)-y(t)\|_X+\|x(t)-z(t)\|_X\bigr)\mathrm{d}t\\
&\quad\leq \biggl((1-s)^2-2s(1-s)c+s^2C\\
&\qquad +C\int_0^1\bigl(\|x(t)-y(t)\|_X+\|x(t)-z(t)\|_X\mathrm{d}t\bigr)\biggr)\langle G_2\mu-G_2\lambda, \mu-\lambda\rangle.
\end{align*}
By choosing $0<s<1$ small enough that
\begin{displaymath}
    L_1=(1-s)^2-2s(1-s)c+s^2C<1
\end{displaymath}
we have, for a fixed constant $C_1>0$,
\begin{align}\label{eq:prebound}
\begin{split}
    \langle G_2K\mu-G_2K\lambda, K\mu-K\lambda\rangle&\leq \biggl(L_1+C_1\int_0^1\|x(t)-y(t)\|_X\\
    &\quad+\|x(t)-z(t)\|_X\mathrm{d}t\biggr)\langle G_2\mu-G_2\lambda, \mu-\lambda\rangle.
\end{split}
\end{align}
Since $\mu$ and $\lambda$ were arbitrary and the constants are independent of $\mu$ and $\lambda$, this holds for all $\mu,\lambda\in W$. It remains to estimate the terms $\|x(t)-y(t)\|_X$ and $\|x(t)-z(t)\|_X$. To do this, first observe that
\begin{align*}
    \|x(t)-\eta\|_X^2&=\|t(K\mu-K\eta)+(1-t)(K\lambda-K\eta)\|_X^2\\
    &\leq C\bigl(\|K\mu-K\eta\|_X^2+\|K\lambda-K\eta)\|_X^2\bigr)\\
    &\leq C\bigl(\|\mu-\eta\|_X^2+\|\lambda-\eta)\|_X^2\bigr)\\
    &\leq C(\langle G_2\mu-G_2\eta, \mu-\eta\rangle+\langle G_2\lambda-G_2\eta, \lambda-\eta\rangle)
\end{align*}
for all $\mu,\lambda\in W$. This and similar observations for $y(t)$ and $z(t)$ yield the following inequalities
\begin{align}\label{eq:estimates}
\begin{split}
\|x(t)-\eta\|_X^2&\leq C(\langle G_2\mu-G_2\eta, \mu-\eta\rangle+\langle G_2\lambda-G_2\eta, \lambda-\eta\rangle)\\
\|y(t)-\eta\|_X^2&\leq C(\langle G_2\mu-G_2\eta, \mu-\eta\rangle+\langle G_2\lambda-G_2\eta, \lambda-\eta\rangle)\\
\|z(t)-\eta\|_X^2&\leq C(\langle G_2\mu-G_2\eta, \mu-\eta\rangle+\langle G_2\lambda-G_2\eta, \lambda-\eta\rangle)
\end{split}
\end{align}
for all $\mu,\lambda\in W$. Now, for some $r>0$ to be chosen later, let
\begin{displaymath}
    D_r(\eta)=\{\mu: \langle G_2\mu-G_2\eta, \mu-\eta\rangle<r^2\}.
\end{displaymath}
We assume that $r>0$ is small enough that $D_r(\eta) \subset W$. This, together with~\cref{eq:estimates}, implies that we have
\begin{displaymath}
    \|x(t)-y(t)\|_X\leq Cr\quad\text{and}\quad\|x(t)-z(t)\|_X\leq Cr,
\end{displaymath}
for all $\mu, \lambda\in D_r(\eta)$. This means that we can choose $r>0$ small enough that
\begin{displaymath}
    C_1(\|x(t)-y(t)\|_X+\|x(t)-z(t)\|_X)\leq \frac{1-L_1}{2}
\end{displaymath}
for all $\mu, \lambda\in D_r(\eta)$. Then, from~\cref{eq:prebound}, we finally obtain
\begin{align*}
\langle G_2K\mu-G_2K\lambda, K\mu-K\lambda\rangle \leq L\langle G_2\mu-G_2\lambda, \mu-\lambda\rangle
\end{align*}
with
\begin{displaymath}
    L=L_1+\frac{1-L_1}{2}=\frac{L_1+1}{2}<1.
\end{displaymath}
Thus if $\eta^{n+1}=K\eta^n$ with $\eta^n\in D_r(\eta)$ then
\begin{displaymath}
    \langle G_2\eta^{n+1}-G_2\eta, \eta^{n+1}-\eta\rangle \leq L\langle G_2\eta^n-G_2\eta, \eta^n-\eta\rangle,
\end{displaymath}
which implies that $\eta^{n+1}\in D_r(\eta)$. If we assume that $\eta^0\in D_r(\eta)$ and iterate the above inequality we get
\begin{align*}
c\|\eta^{n+1}-\eta\|_X^2&\leq\langle G_2\eta^{n+1}-G_2\eta, \eta^{n+1}-\eta\rangle \\
&\leq L^{n+1}\langle G_2\eta^0-G_2\eta, \eta^0-\eta\rangle\leq CL^{n+1}\|\eta^0-\eta\|_X^2.
\end{align*}
Note that, by the Lipschitz continuity of $G_2$, the condition $\eta^0\in D_r(\eta)$ is satisfied if $\eta^0$ is close enough to $\eta$ in $X$.
\end{proof}
 \section{Weak formulations of semilinear elliptic equations}\label{sec:weak}
 To state our weak formulations we require the following assumption.
\begin{assumption}\label{ass:dom}
    The sets $\Omega$ and $\Omega_i$ are bounded Lipschitz domains. The exterior boundary $\partial\Omega\setminus\partial\Omega_i$ and the interface $\Gamma$ are $(d-1)$-dimensional Lipschitz manifolds.
\end{assumption}
With~\cref{ass:dom} we can define the Sobolev spaces
\begin{align*}
    V&=H_0^1(\Omega),\qquad V_i^0=H_0^1(\Omega_i),\\
    V_i&=\{u\in H^1(\Omega_i):\, T_iu=0\},\\
    \Lambda &=\{\mu\in L^2(\Gamma):\, E_i\mu\in H^{1/2}(\partial\Omega_i)\},
\end{align*}
where $T_i:V_i\rightarrow\Lambda$ is the bounded and linear trace operator~\cite[Lemma 4.4]{eeeh22} and $E_i:L^2(\Gamma)\rightarrow L^2(\partial\Omega_i)$ is the extension by $0$. We also introduce the bounded linear extension operator $R_i: \Lambda\rightarrow V_i$ that is a right inverse to $T_i$, i.e., $T_iR_i\eta=\eta$ for all $\eta\in\Lambda$. We use the standard $H^1(\Omega_i)$-norm on $V_i$ and the denote the seminorm by $|u|_{V_i}=\|\nabla u\|_{L^2(\Omega_i)^d}$. On the space $\Lambda$ we use the norm $\|\mu\|_\Lambda=\|E_i\mu\|_{H^{1/2}(\partial\Omega_i)}$. For more details and a definition of $H^{1/2}(\partial\Omega_i)$ we refer to~\cite{eeeh22}. 

We make the following assumptions on the equation~\cref{eq:strong}.
 \begin{assumption}\label{ass:eq}
    Let
    \begin{displaymath}
        p^*=\frac{2d}{d-2},
    \end{displaymath}
    the Sobolev conjugate of $p=2$. If $d=2$ then $p^*$ can be chosen arbitrarily large. We make the following assumptions on $\alpha:\Omega\rightarrow\R$, $\beta:\Omega\times\R\rightarrow\R$, and $f$.
     \begin{itemize}
         \item The function $\alpha$ is in $L^\infty(\Omega)$.
         \item The function $\beta(\cdot, y)$ is measurable for all $y\in\R$ and $\beta(\cdot, 0)\in L^\infty(\Omega)$. Moreover, $\beta$ is differentiable with respect to $y$ with derivative $\beta_y$. The function $\beta_y(\cdot, y)$ is measurable for all $y\in\R$, $\beta_y(\cdot, 0)\in L^\infty(\Omega)$, and
         \begin{equation}\label{eq:betaylip}
             \bigl|\beta_y(x, y)-\beta_y(x, y')\bigr|\leq L_1(y, y')|y-y'|
         \end{equation}
         for all $y, y'\in\R$ and almost every $x\in\Omega$. The Lipschitz constant $L_1$ is assumed to satisfy
         \begin{displaymath}
            L_1(y, y')\leq C\bigl(1+|y|^{p^*-3}+|y'|^{p^*-3}\bigr),\qquad \text{for all }y, y'\in \R.
         \end{displaymath}
         \item We have the linear approximation bound
         \begin{equation}\label{eq:betaapp}
            \bigl|\beta\bigl(x, y\bigr)-\beta\bigl(x', y'\bigr)-\beta_y(x, y')(y-y')\bigr|\leq L_2(y, y')|y-y'|^2
         \end{equation}
         for all $y, y'\in\R$ and almost every $x\in\Omega$. The constant is assumed to satisfy
         \begin{equation}\label{eq:Lpbound}
            L_2(y, y')\leq C(1+|y|^{p^*-3}+|y'|^{p^*-3})\qquad \text{for all }y,y'\in \R.
         \end{equation}
         \item The function $\beta$ satisfies the following monotonicity condition
         \begin{equation}\label{eq:betamono}
             \bigl(\beta(x, y)-\beta(x, y')\bigr)\bigl(y-y'\bigr)\geq  -h(x)|y-y'|^2
         \end{equation}
         for all $y, y'\in\R$ and almost every $x\in\Omega$. Here the function $h$ satisfies
        \begin{equation}\label{eq:monbound}
            \inf_{x\in\Omega} \alpha(x) > C_p \sup_{x\in\Omega} h(x),
        \end{equation}
         where $C_p$ is the maximum of the Poincaré constants of $\Omega$ and $\Omega_i$.
         \item The function $\beta_y$ satisfies the coercivity condition
        \begin{equation}\label{eq:betaycoer}
            \beta_y(x, y)\geq -h(x)\qquad \text{for all }y\in \R \text{ and almost every }x\in\Omega.
        \end{equation}
         Here the function $h$ satisfies~\cref{eq:monbound}.
         \item $f\in V^*$ and there exist $f_i\in V_i^*$ such that
         \begin{displaymath}
            \langle f, v\rangle =\langle f_1, \restr{v}{\Omega_1}\rangle+\langle f_2, \restr{v}{\Omega_2}\rangle.
         \end{displaymath}
     \end{itemize}
 \end{assumption}
    \begin{example}\label{ex:ex1}
        Let $\alpha(x)=1$ and $\beta(x, y) = 10|y|^2y$. The derivative of $\beta$ is
        \begin{displaymath}
            \beta_y(x, y) = 30|y|^2,
        \end{displaymath}
        which gives that
        \begin{displaymath}
            |\beta_y(x, y)-\beta_y(x, y')|\leq 30(|y|+|y'|)|y-y'|
        \end{displaymath}
        and
        \begin{align*}
            &\bigl|\beta\bigl(x, y\bigr)-\beta\bigl(x', y'\bigr)-\beta_y(x, y')(y-y')\bigr|\\
            &\quad= 10|y+2y'||y-y'|^2\leq 20(|y|+|y'|)|y-y'|^2.
        \end{align*}
    \end{example}
    \begin{example}\label{ex:ex2}
        Let $\beta(x, y) = y$. If we replace $\alpha(x)$ in~\cref{eq:strong,eq:dn} with
        \begin{displaymath}
            \alpha(\nabla u)=|\nabla u|
        \end{displaymath}
        we get the $p$-Laplace equation with $p=3$, a degenerate quasilinear equation that does not satisfy~\cref{ass:eq}, but which will be used to demonstrate the generality of the nonlinear Dirichlet--Neumann method.
    \end{example}
    Let $A:V\rightarrow V^*$ and $A_i:V_i\rightarrow V_i^*$ be defined by
    \begin{align*}
        \langle Au, v\rangle&=\int_{\Omega}\alpha(x)\nabla u\cdot\nabla v+\beta(x, u) v \mathrm{d}x \quad\text{and}\\
        \langle A_iu_i, v_i\rangle&=\int_{\Omega_i}\alpha(x)\nabla u_i\cdot\nabla v_i+\beta(x, u_i) v_i \mathrm{d}x,
    \end{align*}
    respectively. The weak formulation of~\cref{eq:strong} is to find $u\in V$ such that
    \begin{equation}\label{eq:weak}
        \langle Au, v\rangle=\langle f, v\rangle\qquad \textrm{for all } v\in V.
    \end{equation}
    We also define the operator $A_i': V_i\rightarrow S(V_i, V_i^*)$
    \begin{align*}
        \langle A_i'(w_i)u_i, v_i\rangle&=\int_{\Omega_i}\alpha(x)\nabla u_i\cdot\nabla v_i+\beta_y(x, w_i) u_iv_i\mathrm{d}x.
    \end{align*}
    The notation $A_i'$ is justified by the following theorem.
    
    \begin{theorem}\label{thm:Ai}
        Suppose that~\cref{ass:dom,ass:eq} holds. The operator $A_i:V_i\rightarrow V_i^*$ is uniformly monotone and Lipschitz continuous on any bounded set $U\subset V_i$. A similar result holds for $A$. In particular, \cref{eq:weak} has a unique solution.
        
        For any $w\in V_i$, the operator $A_i'(w):V_i\rightarrow V_i^*$ is coercive, symmetric, and linear. Moreover, $A_i$ is Fréchet differentiable with Fréchet derivative $A_i': V_i\rightarrow S(V_i, V_i^*)$. The Fréchet derivative $A_i'$ is Lipschitz continuous on any bounded set $U\subset V_i$. 
    \end{theorem}
    \begin{proof}
        First recall the inequality,
        \begin{displaymath}
            (x+y)^r\leq \max(1, 2^{r-1})(x^r+y^r),\quad r,x,y\geq 0,
        \end{displaymath}
        which follows from the convexity of the function $x\mapsto x^r$ when $r\geq1$ and the subadditivity of the same function when $r<1$. This inequality will be used many times throughout the proof without explicit mention. We note that~\cref{eq:betaylip} and $\beta_y(\cdot, 0)\in L^\infty(\Omega)$ implies that
         \begin{equation}\label{eq:betaybound}
             \bigl|\beta_y(x, y)\bigr|\leq\bigl|\beta_y(x, y)-\beta_y(x, 0)\bigr|+\bigl|\beta_y(x, 0)\bigr|\leq L_0(y)
         \end{equation}
         for all $y, y'\in\R$ and almost every $x\in\Omega$. Here $L_0$ satisfies
         \begin{equation}\label{eq:L0bound}
            L_0(y)\leq C\bigl(1+|y|^{p^*-2}\bigr)\qquad \text{for all }y\in \R.
         \end{equation}
         In turn,~\cref{eq:betaybound} implies that
        \begin{displaymath}
            |\beta(x, y)-\beta(x, y')|=\Bigl|\int_{y'}^y\beta_y(x, t)\mathrm{d}t\Bigr|\leq C(1+|y|^{p^*-2}+|y'|^{p^*-2})|y-y'|
        \end{displaymath}
        for all $y, y'\in\R$ and almost every $x\in\Omega$. With this observation and the monotonicity~\cref{eq:betamono}, the fact that $A_i$ and $A$ are uniformly monotone and Lipschitz continuous and that~\cref{eq:weak} has a unique solution follows as in~\cite[Lemma 1]{eeeh24apnum}.

        We now let $w\in V_i$ and consider $A_i'(w)$. We first show that $A_i'(w)\in S(V_i, V_i^*)$. The symmetry is obvious from the definition. To show boundedness let $u,v\in V_i$, recall the Sobolev conjugate $p^*$ (of $p=2$), and let
        \begin{displaymath}
            q'=\frac{p^*}{p^*-2}.
        \end{displaymath}
        Since we have the Sobolev embedding $V_i\hookrightarrow L^{p*}(\Omega_i)$, it follows that $u,v,w\in L^{p*}(\Omega_i)$ and
        \begin{displaymath}
            \|u\|_{L^{p^*}(\Omega_i)}\leq C\|u\|_{V_i},
        \end{displaymath}
        with similar bounds for $v$ and $w$. Moreover, by~\cref{eq:L0bound} we have
        \begin{displaymath}
            |L_0(w)|^{q'}\leq C\bigl(1+|w|^{p^*}\bigr)
        \end{displaymath}
        and therefore $L_0(w)\in L^{q'}(\Omega_i)$ with
        \begin{displaymath}
            \|L_0(w)\|_{L^{q'}(\Omega_i)}\leq C\bigl(1+\|w\|_{L^{p^*}(\Omega_i)}^{p^*-2}\bigr).
        \end{displaymath}
        Using the Hölder inequality twice, first with $q'$ and its Hölder conjugate $p^*/2$ and then with $2$ and its Hölder conjugate $2$, yields
        \begin{align*}
            &\int_{\Omega_i}L_0(w)|u||v| \mathrm{d}x\leq \|L_0(w)\|_{L^{q'}(\Omega_i)}\biggl(\int_{\Omega_i}|u|^{p^*/2}|v| ^{p^*/2}\mathrm{d}x\biggr)^{2/p^*}\\
            &\quad\leq \|L_0(w)\|_{L^{q'}(\Omega_i)}\|u\|_{L^{p^*}(\Omega_i)}\|v\|_{L^{p^*}(\Omega_i)}.
        \end{align*}
        With these observations, we have from~\cref{eq:betaybound} that
        \begin{align*}
            &\bigl|\langle A_i'(w)u, v\rangle\bigr|=\biggl|\int_{\Omega_i}\alpha(x)\nabla u\cdot \nabla v+\beta_y(x, w)uv \mathrm{d}x\biggr|\\
            &\quad\leq \|\alpha\|_{L^\infty(\Omega_i)}\|u\|_{V_i}\|v\|_{V_i}+\|L_0(w)\|_{L^{q'}(\Omega_i)}\|u\|_{L^{p*}(\Omega_i)}\|v\|_{L^{p*}(\Omega_i)}\\
            &\quad\leq \|\alpha\|_{L^\infty(\Omega_i)}\|u\|_{V_i}\|v\|_{V_i}+C(1+\|w\|_{L^{p^*}(\Omega_i)})^{p^*-2}\|u\|_{L^{p*}(\Omega_i)}\|v\|_{L^{p*}(\Omega_i)}\\
            &\quad\leq \|\alpha\|_{L^\infty(\Omega_i)}\|u\|_{V_i}\|v\|_{V_i}+C(1+\|w\|_{V_i})^{p^*-2}\|u\|_{V_i}\|v\|_{V_i}.
        \end{align*}
        Thus we have shown that $A_i'(w)$ is bounded. The coercivity of $A_i'(w)$ follows from~\cref{eq:monbound,eq:betaycoer} since
        \begin{align*}
            \langle A_i'(w)u, u\rangle &=\int_{\Omega_i}\alpha(x)|\nabla u|^2+\beta_y(x, w)|u|^2 \mathrm{d}x\geq\inf_{x\in\Omega_i} \alpha(x)|u|_{V_i}^2-\sup_{x\in\Omega_i}h(x)\|u\|_{L^2(\Omega_i)}^2\\
            &\geq c\|u\|_{V_i}^2.
        \end{align*}
        
        We now show that $A'_i$ is the Fréchet derivative of $A_i$. Let $u,v,h\in V_i$ such that $\|h\|_{V_i}\rightarrow 0$. The argument is similar as above, but we let
        \begin{displaymath}
            q=\frac{p^*}{p^*-3}.
        \end{displaymath}
        It follows that $u,v,h\in L^{p*}(\Omega_i)$ with appropriate bounds. Moreover, by~\cref{eq:Lpbound} we have
        \begin{displaymath}
            |L_2(u+h, u)|^q\leq C\bigl(1+|u+h|^{p^*}+|u|^{p^*}\bigr)
        \end{displaymath}
        and therefore $L_2(u+h, u)\in L^q(\Omega_i)$ with
        \begin{displaymath}
            \|L_2(u+h, u)\|_{L^q(\Omega_i)}\leq C\bigl(1+\|u+h\|_{L^{p^*}(\Omega_i)}^{p^*-3}+\|u\|_{L^{p^*}(\Omega_i)}^{p^*-3}\bigr).
        \end{displaymath}
        
        Using the Hölder inequality twice again, this time first with $q$ and its Hölder conjugate $p^*/3$ and then with $3$ and its Hölder conjugate $3/2$, yields
        \begin{align*}
            &\int_{\Omega_i}L_2(u+h, u)|h|^2|v| \mathrm{d}x\leq \|L_2(u+h, u)\|_{L^q(\Omega_i)}\biggl(\int_{\Omega_i}|h|^{2p^*/3}|v| ^{p^*/3}\mathrm{d}x\biggr)^{3/p^*}\\
            &\quad\leq \|L_2(u+h, u)\|_{L^q(\Omega_i)}\|v\|_{L^{p^*}(\Omega_i)}\biggl(\int_{\Omega_i}|h| ^{p^*}\mathrm{d}x\biggr)^{2/p^*}\\
            &\quad\leq \|L_2(u+h, u)\|_{L^q(\Omega_i)}\|v\|_{L^{p^*}(\Omega_i)}\|h\|_{L^{p^*}(\Omega_i)}^2.
        \end{align*}
        These observations together with~\cref{eq:betaapp} gives the estimate
        \begin{align*}
            &\|A_i(u+h)-A_i(u)-A_i'(u)h\|_{V_i^*}=\sup_{\|v\|_{V_i}=1}\bigl|\langle A_i(u+h)-A_i(u)-A_i'(u)h, v\rangle\bigr|\\
            &\quad= \sup_{\|v\|_{V_i}=1}\biggl|\int_{\Omega_i}\bigl(\beta(x, u+h)-\beta(x, u)-\beta_y(x, u)h\bigr) v  \mathrm{d}x\biggr|\\
            &\quad\leq \sup_{\|v\|_{V_i}=1}\int_{\Omega_i}L_2(u+h, u)|h|^2|v| \mathrm{d}x\\
            &\quad\leq \sup_{\|v\|_{V_i}=1}\|L_2(u+h, u)\|_{L^q(\Omega_i)}\|h\|_{L^{p^*}(\Omega_i)}^2\|v\|_{L^{p^*}(\Omega_i)}\\
            &\quad\leq \sup_{\|v\|_{V_i}=1}C\bigl(1+\|u+h\|_{L^{p^*}(\Omega_i)}^{p^*-3}+\|u\|_{L^{p^*}(\Omega_i)}^{p^*-3}\bigr)\|h\|_{L^{p^*}(\Omega_i)}^2\|v\|_{L^{p^*}(\Omega_i)}\\
            &\quad\leq \sup_{\|v\|_{V_i}=1}C\bigl(1+\|u\|_{L^{p^*}(\Omega_i)}^{p^*-3}+\|h\|_{L^{p^*}(\Omega_i)}^{p^*-3}\bigr)\|h\|_{L^{p^*}(\Omega_i)}^2\|v\|_{L^{p^*}(\Omega_i)}\\
            &\quad\leq \sup_{\|v\|_{V_i}=1}C(1+\|u\|_{V_i}^{p^*-3}+\|h\|_{V_i}^{p^*-3}\bigr)\|h\|_{V_i}^2\|v\|_{V_i}\\
            &\quad\leq C(1+\|u\|_{V_i}^{p^*-3}+\|h\|_{V_i}^{p^*-3}\bigr)\|h\|_{V_i}^2.
        \end{align*}
        Dividing by $\|h\|_{V_i}$ and letting $h\rightarrow 0$ yields the appropriate limit. It remains to show that $A'_i$ is Lipschitz continuous. Let $w, w'\in V_i$. Arguing similarly as above using~\cref{eq:betaylip} shows that
        \begin{align*}
            &\bigl|\bigl\langle \bigl(A'_i(w)-A'_i(w')\bigr)u, v\bigr\rangle\bigr|\\
            &\quad=\biggl|\int_{\Omega_i}\bigl(\beta_y(x, w)-\beta_y(x, w')\bigr)u v  \mathrm{d}x\biggr|\\
            &\quad\leq\int_{\Omega_i}L_1(w, w')|w-w'||u||v| \mathrm{d}x\\
            &\quad\leq C\|L_1(w, w')\|_{L^q(\Omega_i)}\|w-w'\|_{L^{p^*}(\Omega_i)}\|u\|_{L^{p^*}}\|v\|_{L^{p^*}(\Omega_i)}\\
            &\quad\leq C(1+\|w\|_{L^{p^*}(\Omega_i)}+\|w'\|_{L^{p^*}(\Omega_i)})^{p^*-3}\|w-w'\|_{L^{p^*}(\Omega_i)}\|u\|_{L^{p^*}(\Omega_i)}\|v\|_{L^{p^*}(\Omega_i)}\\
            &\quad\leq C(1+\|w\|_{V_i}+\|w'\|_{V_i})^{p^*-3}\|w-w'\|_{V_i}\|u\|_{V_i}\|v\|_{V_i}
        \end{align*}
        for all $u, v\in V_i$. This gives
        \begin{align*}
            &\|A'_i(w)-A'_i(w')\|_{B(V_i, V_i^*)}=\sup_{\|u\|_{V_i}=1}\bigl\|\bigl(A'_i(w)-A'_i(w')\bigr)u\bigr\|_{V_i^*}\\
            &\quad=\sup_{\|u\|_{V_i}=\|v\|_{V_i}=1}\bigl|\bigl\langle \bigl(A'_i(w)-A'_i(w')\bigr)u, v\bigr\rangle\bigr|\\
            &\quad\leq \sup_{\|u\|_{V_i}=\|v\|_{V_i}=1}C(1+\|w\|_{V_i}+\|w'\|_{V_i})^{p^*-3}\|w-w'\|_{V_i}\|u\|_{V_i}\|v\|_{V_i}\\
            &\quad= C(1+\|w\|_{V_i}+\|w'\|_{V_i})^{p^*-3}\|w-w'\|_{V_i},
        \end{align*}
        which shows that $A_i'$ is Lipschitz continuous on any bounded set.

    \end{proof}
    
    \section{Transmission problem and Steklov--Poincaré operators}\label{sec:ddm}
    The transmission problem is to find $(u_1, u_2)\in V_1\times V_2$ such that
    \begin{equation}\label{eq:weaktran}
    	\left\{\begin{aligned}
    	     \langle A_iu_i, v_i\rangle&=\langle f_i, v_i\rangle & & \text{for all } v_i\in V_i^0,\, i=1,2,\\
    	     T_1u_1&=T_2u_2, & &\\
    	     \textstyle\sum_{i=1}^2 \langle A_i& u_i, R_i\mu\rangle-\langle f_i, R_i\mu\rangle=0 & &\text{for all }\mu\in \Lambda. 
    	\end{aligned}\right.
    \end{equation}
    The transmission problem is equivalent to the weak equation~\cref{eq:weak}, see, e.g.,~\cite[Theorem 5.2]{eeeh22} for a proof.

    Before introducing the Steklov--Poincaré operators we must define the solution operators. For $\eta\in\Lambda$ let $u_i=F_i\eta$ denote the solution to the problem
    \begin{equation}\label{eq:Fi}
    	\left\{
    	\begin{aligned}
    		\langle A_iu_i-f_i, v\rangle &= 0 \qquad&\text{for all } v\in V_i^0,\\
    		T_iu_i&=\eta. &
    	\end{aligned}
    	\right.
    \end{equation}
    Similarly, for $\eta,\nu\in\Lambda$ let $u_i=F'_i(\nu)\eta$ denote the solution to the problem
    \begin{equation}\label{eq:hatFi}
    	\left\{
    	\begin{aligned}
    		\langle A'_i(F_i\nu)u_i, v\rangle &= 0 \qquad&\text{for all } v\in V_i^0,\\
    		T_iu_i&=\eta. &
    	\end{aligned}
    	\right.
    \end{equation}
    It has been proven that both of these operators exists under weaker assumptions, see~\cite[Lemma 6.1]{eeeh22}. The notation $F'$ is justified by the following theorem.
    \begin{theorem}\label{thm:Fi}
        Suppose that~\cref{ass:eq,ass:dom} hold. The operators $F_i:\Lambda\rightarrow V_i$ are Lipschitz continuous on any bounded subset $U\subset \Lambda$. Moreover, they are Fréchet differentiable with Fréchet derivative $F'_i:\Lambda\rightarrow B(\Lambda, V_i)$ defined through~\cref{eq:hatFi}. The operators $F'_i$ are Lipschitz continuous on any bounded subset $U\subset \Lambda$.
    \end{theorem}
    \begin{proof}
        For the proof of fact that $F_i$ is Lipschitz continuous we refer to~\cite[Lemma 2]{eeeh24apnum}.

        We first show that $F'_i$ is bounded, i.e., that it maps bounded sets to bounded sets. Let $U$ be a bounded subset $U\subset \Lambda$ . Since $F_i$ is Lipschitz on $U$ it is also bounded on $U$, which means that $F_i(U)$ is a bounded set. Also, for the same reason, $A_i'$ is bounded on $F_i(U)$. It follows that
        \begin{displaymath}
            \|A_i'(F_i\nu)F_i'(\nu)\eta\|_{V_i^*}\leq C\|F_i'(\nu)\eta\|_{V_i}\quad\text{for all }\nu\in U, \eta\in\Lambda.
        \end{displaymath}
        Thus, by the boundedness of $R_i$, the coercivity of $A_i'(F_i\nu)$, and the fact that $F_i'(\nu)\eta-R_i\eta\in V_i^0$, we have
        \begin{align*}
            c\|F_i'(\nu)\eta\|_{V_i}^2&\leq\langle A_i'(F_i\nu)F_i'(\nu)\eta, F_i'(\nu)\eta\rangle\\
            &=\langle A_i'(F_i\nu)F_i'(\nu)\eta, R_i\eta\rangle+\langle A_i'(F_i\nu)F_i'(\nu)\eta, F_i'(\nu)\eta-R_i\eta\rangle\\
            &=\langle A_i'(F_i\nu)F_i'(\nu)\eta, R_i\eta\rangle\\
            &\leq \|A_i'(F_i\nu)F_i'(\nu)\eta\|_{V_i^*}\|R_i\eta\|_{V_i}\\
            &\leq C\|F_i'(\nu)\eta\|_{V_i}\|\eta\|_{\Lambda}
        \end{align*}
        for all $\nu\in U, \eta\in\Lambda$.
        It follows that
        \begin{displaymath}
            \|F_i'(\nu)\|_{B(\Lambda, V_i)}\leq C
        \end{displaymath}
        for all $\nu\in U$. Thus we have shown that $F_i'$ is bounded.

        We now verify that $F_i'$ is the Fréchet derivative of $F_i$. Let $\nu, h\in \Lambda$ and introduce $u=F_i\nu$, $k=F_i(\nu+h)-F_i\nu$, and $w=F_i(\nu+h)-F_i\nu-F_i'(\nu)h$. Note, in particular, that $w\in V_i^0$ so that
        \begin{displaymath}
            \Bigl\langle A_i(u+k)-A_i(u), w\Bigr\rangle=0\quad\text{and}\quad\Bigl\langle A'_i(F_i\nu)F_i'(\nu)h, w\Bigr\rangle=0.
        \end{displaymath}
        Therefore, by the coercivity of $A_i'(F_i\nu)$,
        \begin{align*}
            &\|F_i(\nu+h)-F_i\nu-F_i'(\nu)h\|_{V_i}^2\leq\Bigl\langle A'_i(F_i\nu)\bigl(F_i(\nu+h)-F_i\nu-F_i'(\nu)h\bigr), w\Bigr\rangle\\
            &\quad=\Bigl\langle A'_i(F_i\nu)\bigl(F_i(\nu+h)-F_i\nu\bigr), w\Bigr\rangle=\langle A'_i(u)k, w\rangle\\
            &\quad=-\langle A_i(u+k)-A_i(u)-A'_i(u)k, w\rangle\\
            &\quad\leq\|A_i(u+k)-A_i(u)-A'_i(u)k\|_{V_i^*}\|w\|_{V_i}\\
            &\quad=\frac{\|A_i(u+k)-A_i(u)-A'_i(u)k\|_{V_i^*}}{\|k\|_{V_i}}\|k\|_{V_i}\|F_i(\nu+h)-F_i\nu-F_i'(\nu)h\|_{V_i}.
        \end{align*}
        We divide by $\|h\|_{\Lambda}\|F_i(\nu+h)-F_i\nu-F_i'(\nu)h\|_{V_i}$ and let $\|h\|_{\Lambda}\rightarrow 0$. Since $F_i$ is Lipschitz continuous in a neighbourhood around $\nu$ this implies that $\|k\|_{V_i}\rightarrow 0$ and that $\|k\|_{V_i}/\|h\|_{\Lambda}$ is bounded. Since $A_i$ is Fréchet differentiable, we get
        \begin{align*}
            \frac{\|F_i(\nu+h)-F_i\nu-F_i'(\nu)h\|_{V_i}}{\|h\|_{\Lambda}}\leq \frac{\|A_i(u+k)-A_i(u)-A'_i(u)k\|_{V_i^*}}{\|k\|_{V_i}}\frac{\|k\|_{V_i}}{\|h\|_{\Lambda}}\rightarrow 0
        \end{align*}
        as $\|h\|_{\Lambda}\rightarrow 0$. We now show that $F_i'$ is Lipschitz continuous on any bounded set. To show this we first note that $F'_i(\nu)\eta-F'_i(\lambda)\eta\in V_i^0$, which yields the identity
        \begin{align*}
            &\Bigl\langle A'_i(F_i\nu)\bigl(F'_i(\nu)-F'_i(\lambda)\bigr)\eta, F'_i(\nu)\eta-F'_i(\lambda)\eta\Bigr\rangle\\
            &\quad=\Bigl\langle 0-A'_i(F_i\nu)F'_i(\lambda)\eta, F'_i(\nu)\eta-F'_i(\lambda)\eta\Bigr\rangle\\
            &\quad=\Bigl\langle \Bigl(A'_i(F_i\lambda)-A'_i(F_i\nu)\Bigr)F'_i(\lambda)\eta, F'_i(\nu)\eta-F'_i(\lambda)\eta\Bigr\rangle
        \end{align*}
        for all $\nu,\lambda\in\Lambda$. Now consider any bounded subset $U\subset\Lambda$. Then $F_i$ is Lipschitz continuous on $U$. In particular, $F_i(U)$ is a bounded set. Therefore, $A_i'$ is Lipschitz continuous on $F_i(U)$. The Lipschitz continuity now follows from these observations and the coercivity of $A'_i(F_i\nu)$ since
        \begin{align*}
            c\|F'_i(\nu)\eta-F'_i(\lambda)\eta\|_{V_i}^2&\leq\Bigl\langle A'_i(F_i\nu)\bigl(F'_i(\nu)-F'_i(\lambda)\bigr)\eta, F'_i(\nu)\eta-F'_i(\lambda)\eta\Bigr\rangle\\
            &=\Bigl\langle \bigl(A'_i(F_i\lambda)-A'_i(F_i\nu)\bigr)F'_i(\lambda)\eta, F'_i(\nu)\eta-F'_i(\lambda)\eta\Bigr\rangle\\
            &\leq C\|\lambda-\nu\|_\Lambda\|F'_i(\lambda)\eta\|_{V_i}\|F'_i(\nu)\eta-F'_i(\lambda)\eta\|_{V_i}
        \end{align*}
        for all $\nu,\lambda\in U$ and $\eta\in\Lambda$. Dividing by $\|F'_i(\nu)\eta-F'_i(\lambda)\eta\|_{V_i}$ finishes the proof since we have shown that $F'_i$ is bounded.
    \end{proof}
    
    The Steklov--Poincaré operators $S_i:\Lambda\to\Lambda^*$ are defined as 
    \begin{align*}
        \langle S_i\eta, \mu\rangle=\langle A_iF_i\eta-f_i, R_i\mu\rangle.
    \end{align*}
    We also define $S':\Lambda\to B(\Lambda,\Lambda^*)$ as
    \begin{align*}
        \langle S_i'(\nu)\eta, \mu\rangle=\langle A_i'(F_i\nu)F_i'(\nu)\eta, R_i\mu\rangle.
    \end{align*}
    Moreover, we introduce $S=S_1+S_2$. The Steklov--Poincaré equation is then to find $\eta\in\Lambda$ such that
    \begin{equation}\label{eq:sp}
        \langle S\eta, \mu\rangle=0\quad\text{for all }\mu\in\Lambda.
    \end{equation}
    The Steklov--Poincaré equation is equivalent to the transmission problem by the formulas $\eta=T_iu_i$ and $(u_1, u_2)=(F_1\eta, F_2\eta)$, see~\cite[Lemma 6.2]{eeeh22} for details.
    \begin{theorem}\label{thm:sp}
        Suppose that~\cref{ass:eq,ass:dom} hold. The Steklov--Poincaré operators $S_i:\Lambda\rightarrow\Lambda^*$ are Fréchet differentiable with Fréchet derivatives $S_i':\Lambda\rightarrow S(\Lambda, \Lambda^*)$. The operators $S,S_i,S_i'$ are Lipschitz continuous on any bounded set $U\subset \Lambda$. Moreover, the operators $S, S_i$ are uniformly monotone.
    \end{theorem}
    \begin{proof}
        First let $\nu\in \Lambda$ and consider $S_i'(\nu)$. That $S_i'(\nu)\in B(\Lambda, \Lambda^*)$ follows from
        \begin{align*}
            \bigl|\langle S_i'(\nu)\eta, \mu\rangle\bigr|=\bigl|\langle A_i'(F_i\nu)F_i'(\nu)\eta, R_i\mu\rangle\bigr|\leq\|A_i'(F_i\nu)F_i'(\nu)\eta\|_{V_i}\|R_i\mu\|_{V_i}\leq C\|\eta\|_\Lambda\|\mu\|_\Lambda.
        \end{align*}
        Moreover, by the symmetry of $A_i'(F_i\nu)$ and the fact that $R_i\mu-F_i'(\nu)\mu\in V_i^0$ we have
        \begin{align*}
             \langle S_i'(\nu)\eta, \mu\rangle&=\langle A_i'(F_i\nu)F_i'(\nu)\eta, R_i\mu\rangle\\
             &=\langle A_i'(F_i\nu)F_i'(\nu)\eta, F_i'(\nu)\mu\rangle+\langle A_i'(F_i\nu)F_i'(\nu)\eta, R_i\mu-F_i'(\nu)\mu\rangle\\
             &=\langle A_i'(F_i\nu)F_i'(\nu)\eta, F_i'(\nu)\mu\rangle\\
             &=\langle A_i'(F_i\nu)F_i'(\nu)\mu, F_i'(\nu)\eta\rangle\\
             &=\langle A_i'(F_i\nu)F_i'(\nu)\mu, R_i\eta\rangle\\
             &=\langle S_i'(\nu)\mu, \eta\rangle,
        \end{align*}
        which shows that $S_i'(\nu)\in S(\Lambda, \Lambda^*)$. 
    
        Now note that $S_i=R_i^*(A_iF_i-f_i)$, where $R_i^*:V_i^*\rightarrow \Lambda^*$ is the bounded linear operator that maps $\langle R_i^*u, \mu\rangle=\langle u, R_i\mu\rangle$. From~\cref{thm:Ai,thm:Fi} we have that $S_i$ is Lipschitz continuous and the same holds for $S$. We also have from~\cref{thm:Ai,thm:Fi} and the chain rule~\cite[Lemma 9.6]{teschl} that $S_i$ is Fréchet differentiable with Fréchet derivative $R_i^*A_i'(F_i\nu)F_i'(\nu)=S_i'(\nu)$. Moreover, the monotonicity of $S_i, S$ follows as in~\cite[Theorem 1]{eeeh24apnum}.
        
        It remain to show that $S_i'$ is Lipschitz continuous on any bounded subset $U\subset\Lambda$. We first note that $F_i$ and $F_i'$ are Lipschitz continuous on $U$ and $A_i'$ is Lipschitz continuous on the bounded set $F_i(U)$. Therefore
        \begin{align*}
            &\bigl|\bigl\langle\bigl(S'_i(\nu_1)-S'_i(\nu_2)\bigr)\eta, \mu\bigr\rangle\bigr|=\bigl|\bigl\langle A_i'(F_i\nu_1)F_i'(\nu_1)\eta-A_i'(F_i\nu_2)F_i'(\nu_2)\eta, R_i\mu\bigr\rangle\bigr|\\
            &\quad\leq \bigl\|A_i'(F_i\nu_1)F_i'(\nu_1)\eta-A_i'(F_i\nu_2)F_i'(\nu_2)\eta\bigr\|_{V_i^*}\|R_i\mu\|_{V_i}\\
            &\quad=\bigl\|A_i'(F_i\nu_1)\bigl(F_i'(\nu_1)-F_i'(\nu_2)\bigr)\eta-\bigl(A_i'(F_i\nu_2)-A_i'(F_i\nu_1)\bigr)F_i'(\nu_2)\eta\bigr\|_{V_i^*}\|R_i\mu\|_{V_i}\\
            &\quad\leq\Bigl(\bigl\|A_i'(F_i\nu_1)\bigl(F_i'(\nu_1)-F_i'(\nu_2)\bigr)\eta\bigr\|_{V_i^*}\\
            &\qquad+\bigl\|\bigl(A_i'(F_i\nu_2)-A_i'(F_i\nu_1)\bigr)F_i'(\nu_2)\eta\bigr\|_{V_i^*}\Bigr)\|R_i\mu\|_{V_i}\\
            &\quad\leq\Bigl(C\|F_i'(\nu_1)\eta-F_i'(\nu_2)\eta\|_{V_i^*}+\|F_i\nu_2-F_i\nu_1\|_{V_i}\|F_i'(\nu_2)\eta\|_{V_i}\Bigr)\|R_i\mu\|_{V_i}\\
            &\quad\leq\Bigl(C\|\nu_1-\nu_2\|_{\Lambda}\|\eta\|_\Lambda+\|\nu_2-\nu_1\|_{\Lambda}\|\eta\|_\Lambda\Bigr)\|\mu\|_{\Lambda}\\
            &\quad\leq C\|\nu_1-\nu_2\|_{\Lambda}\|\eta\|_\Lambda\|\mu\|_{\Lambda}
        \end{align*}
        for all $\nu_1,\nu_2\in U$ and $\eta,\mu\in\Lambda$. This implies that
        \begin{align*}
            &\|S'_i(\nu_1)-S'_i(\nu_2)\|_{S(\Lambda, \Lambda^*)}=\sup_{\eta\in\Lambda\setminus\{0\}}\frac{\bigl\|\bigl(S'_i(\nu_1)-S'_i(\nu_2)\bigr)\eta\bigr\|_{\Lambda^*}}{\|\eta\|_\Lambda}\\
            &\quad=\sup_{\eta, \mu\in\Lambda\setminus\{0\}}\frac{\bigl|\bigl\langle\bigl(S'_i(\nu_1)-S'_i(\nu_2)\bigr)\eta, \mu\bigr\rangle\bigr|}{\|\eta\|_\Lambda\|\mu\|_\Lambda}\leq C\|\nu_1-\nu_2\|_{\Lambda}
        \end{align*}
        for all $\nu_1,\nu_2\in U$, which means that $S_i'$ is Lipschitz continuous on $U$. 
    \end{proof}
    \section{Convergence of the Dirichlet--Neumann method}\label{sec:dn}
    The weak formulation of the Dirichlet--Neumann method is to find $(u_1^{n+1}, u_2^{n+1})\in V_1\times V_2$ for $n=0,1,2,\dots$ such that
    \begin{equation}\label{eq:dnweak}
    	\left\{\begin{aligned}
    		&\langle A_1u_1^{n+1}, v_1\rangle=\langle f_1, v_1\rangle & &\text{for all } v_1\in V_1^0,\\
    		&T_1u_1^{n+1}=\eta^n,\\[5pt]
    		&\langle A_2u_2^{n+1}, v_2\rangle=\langle f_2, v_2\rangle & &\text{for all } v_2\in V_2^0,\\
    		&\langle A_2u_2^{n+1}, R_2\mu\rangle-\langle f_2, R_2\mu\rangle+\langle A_1u_1^{n+1}, R_1\mu\rangle & &\\
    		&\quad-\langle f_1, R_1\mu\rangle=0 & &\text{for all }\mu\in \Lambda,
    	\end{aligned}\right.
    \end{equation}
    with $\eta^{n+1}=sT_2u_2^{n+1}+(1-s)\eta^n$. The interface iteration corresponding to the Dirichlet--Neumann method is to find $\eta^{n+1}\in\Lambda$ for $n=0,1,2,\dots$ such that
    \begin{equation}\label{eq:dnsp}
        \eta^{n+1} = (1-s)\eta^n+sS_2^{-1}(0-S_1\eta^n).
    \end{equation}
    The equivalence of the Dirichlet--Neumann method and the interface iteration is stated precisely in the following lemma.
    \begin{lemma}\label{lem:dnsp}
        Suppose that~\cref{ass:eq,ass:dom} hold. The Dirichlet--Neumann method is equivalent to the interface iteration in the following way. If $\{\eta^n\}_{n\geq 0}$ solves~\cref{eq:dnsp} then $\{(u_1^{n+1}, u_2^{n+1})\}_{n\geq 0}$ defined by
        \begin{align*}
            u_1^{n+1}&=F_1\eta^n\\
            u_2^{n+1}&=F_2\biggl(\frac{\eta^{n+1}-(1-s)\eta^n}{s}\biggr)
        \end{align*}
        solves~\cref{eq:dnweak}. Conversely, if $\{(u_1^{n+1}, u_2^{n+1})\}_{n\geq 0}$ solves~\cref{eq:dnweak} then $\{\eta^n=T_iu_1^{n+1}\}_{n\geq 0}$ solves~\cref{eq:dnsp}.
    \end{lemma}
    \begin{proof}
        The first two equations of~\cref{eq:dnweak} are equivalent to $u_1^{n+1}=F_1\eta^n$ and the third equation of~\cref{eq:dnweak} together with the update equation $\eta^{n+1}=sT_2u_2^{n+1}+(1-s)\eta^n$ is equivalent to
        \begin{displaymath}
            u_2^{n+1}=F_2\biggl(\frac{\eta^{n+1}-(1-s)\eta^n}{s}\biggr).
        \end{displaymath}

        If we write out~\cref{eq:dnsp} using $S_i=R_i^*A_iF_i-R_i^*f_i$ we find that it is equivalent to
        \begin{displaymath}
            R_2^*A_2F_2\biggl(\frac{\eta^{n+1}-(1-s)\eta^n}{s}\biggr)-R_2^*f_2=R_1^*A_1F_1\eta^n-R_1^*f_1 \quad\text{in }\Lambda^*,
        \end{displaymath}
        which is exactly the last equation of~\cref{eq:dnweak}.
    \end{proof}
    We finally end up with the following convergence result, which is a direct consequence of~\cref{thm:abstract,thm:sp,lem:dnsp}.
    \begin{corollary}
        Suppose that~\cref{ass:eq,ass:dom} hold. Let $s>0$ be small enough and $\eta^0\in\Lambda$ close enough to $\eta$, the solution to the Steklov--Poincaré equation~\cref{eq:sp}. The iteration~\cref{eq:dnsp} converges linearly to $\eta$ in $\Lambda$. Moreover, the iterates of the Dirichlet--Neumann method~\cref{eq:dnweak} converges linearly to $(u_1, u_2)$, the solution of~\cref{eq:weaktran}, in $V_1\times V_2$.
    \end{corollary}
    \begin{proof}
        From~\cref{thm:abstract,thm:sp} it follows that $\eta^n$ converges linearly to $\eta$, i.e,
        \begin{displaymath}
        	   \|\eta^n-\eta\|_\Lambda\leq C L^n\|\eta^0-\eta\|_\Lambda.
        \end{displaymath}
        In particular, it follows that $\eta^n$ is bounded so there exists a convex and bounded set $U$ that contains $\eta^n$ and $\eta$. Since $F_1$ is Lipschitz continuous on $U$ we have
        \begin{displaymath}
            \|u_1^{n+1}-u_1\|_{V_i}=\|F_1\eta^n-F_1\eta\|_{V_i}\leq C\|\eta^n-\eta\|_\Lambda\leq C L^n\|\eta^0-\eta\|_\Lambda.
        \end{displaymath}
        Since $F_2$ is also Lipschitz continuous on $U$ we have
        \begin{align*}
            &\|u_2^{n+1}-u_2\|_{V_i}=\biggl\|F_2\biggl(\frac{\eta^{n+1}-(1-s)\eta^n}{s}\biggr)-F_2\eta\biggr\|_{V_i}\leq C\biggl\|\frac{\eta^{n+1}-(1-s)\eta^n}{s}-\eta\biggr\|_\Lambda\\
            &\quad= \frac{C}{s}\|(\eta^{n+1}-\eta)-(1-s)(\eta^n-\eta)\|_\Lambda\leq C\bigl(\|\eta^{n+1}-\eta\|_\Lambda+\|\eta^n-\eta)\|_\Lambda\bigr)\\
            &\quad\leq C L^n\|\eta^0-\eta\|_\Lambda.
        \end{align*}
        Therefore, $(u_1^n, u_2^n)$ converges linearly to $(u_1, u_2)$ in $V_1\times V_2$.
    \end{proof}
    \section{Numerical results}\label{sec:num}
    For our numerical results we consider~\cref{ex:ex1,ex:ex2} on the domain $\Omega=(0, 3)\times (0, 2)$ and
        \begin{displaymath}
            f(x, y) = x(3-x)y(2-y).
        \end{displaymath}
        We introduce the Lipschitz domain decomposition as illustrated in~\cref{fig:L} and discretize using the finite element spaces $V^h\subset V$, $V_i^h\subset V_i$ of linear finite elements with mesh size $h$. While we will not discuss it further here, it is possible to construct a discrete variant of our analysis and prove convergence in the finite element space $V_i^h$, see, e.g.~\cite[Section 8]{eeeh24apnum}.
    
    \begin{figure}
        \centering
        \includegraphics[width=0.45\linewidth]{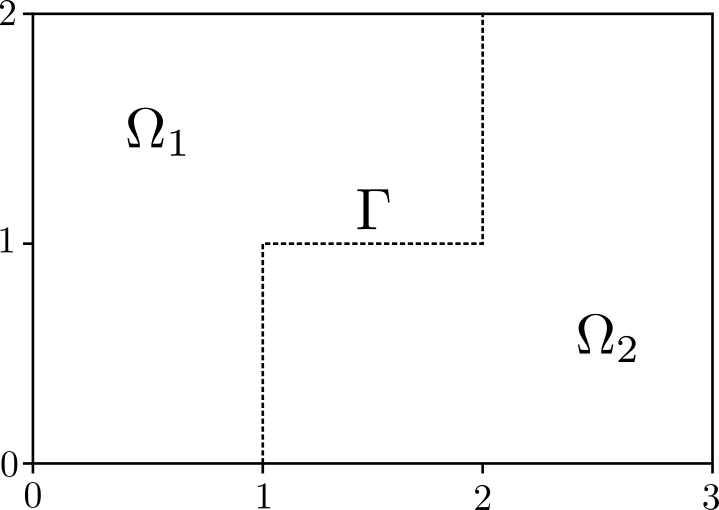}\qquad
        \includegraphics[width=0.45\linewidth]{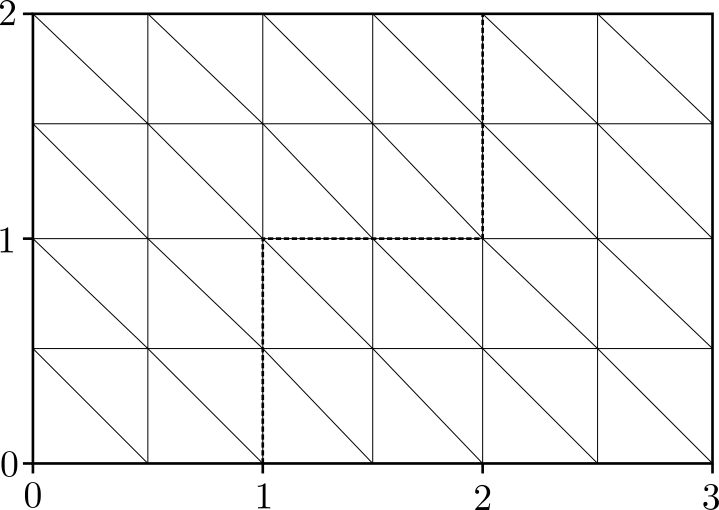}
        \caption{The domain and domain decomposition used in our numerical example (left) and the mesh exemplified with $h=1/2$ (right).}
        \label{fig:L}
    \end{figure}%

    We compute the iteration for three meshes with the mesh sizes $h=(1/64, 1/128, 1/256)$. We use the parameter $s=0.36$ for~\cref{ex:ex1} and $s=0.31$ for~\cref{ex:ex1} since these have been found to be close to optimal. For each mesh we calculate the relative error of our iteration $u_i^{n,h}$ compared to the solutions $u_i^h$ of the finite element problem on $\Omega$ restricted to the subdomains $\Omega_i$. That is,
    \begin{displaymath}
        e=\frac{\|u_1^{n, h}-u_1^h\|_{V_1}+\|u_2^{n, h}-u_2^h\|_{V_2}}{\|u_1^h\|_{V_1}+\|u_2^h\|_{V_2}}.
    \end{displaymath}
    For both examples, we plot the error at each iteration in~\cref{fig:err}. From this we see that the iteration converges linearly, which we have proven for~\cref{ex:ex1}. However, we also see that the same holds for~\cref{ex:ex2}, but with a slightly slower convergence rate. Moreover, we find that in both cases the convergence is independent of the mesh size $h$.
    
    We also compare the Dirichlet--Neumann method with the Neumann--Neumann method~\cite{eeeh24apnum} and the Robin--Robin method~\cite{eeeh22} on the same mesh with mesh size $h=1/256$. For the Robin--Robin method the parameter is $s=46$, which have been chosen since it is near optimal and for the Neumann--Neumann the parameter is $s_1=s_2=0.02$ since it minimizes the error at which the method stagnates. The results are in~\cref{fig:err2}. We find that the Dirichlet--Neumann method converges faster than the Robin--Robin method for both~\cref{ex:ex1} and~\cref{ex:ex2}. The Neumann--Neumann method performs better for~\cref{ex:ex1}, but does not converge at all for~\cref{ex:ex2}.
    \begin{figure}
    \centering
    \includegraphics[width=0.45\linewidth]{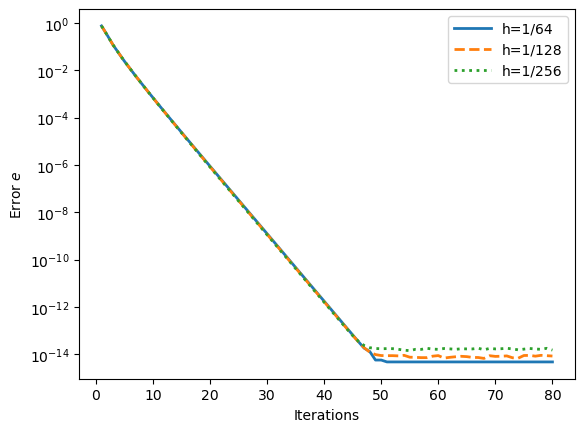}
    \includegraphics[width=0.45\linewidth]{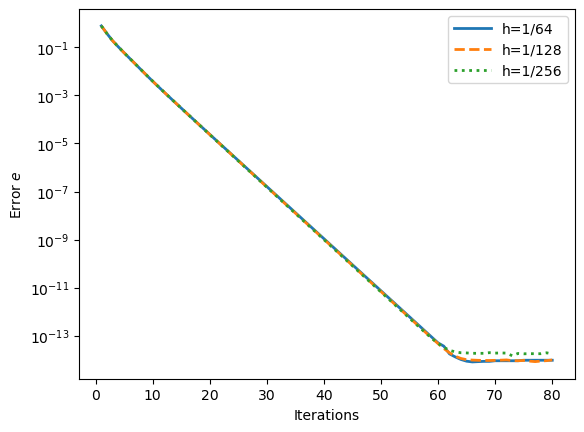}
    \caption{The error of the Dirichlet--Neumann applied to~\cref{ex:ex1} (left) and~\cref{ex:ex2} (right) together with the decomposition as in~\cref{fig:L}. Three different values of the mesh parameter $h$ is used, but for most of the iterations, they all have the same error and therefore only one can be seen.}
    \label{fig:err}
\end{figure}%

    \begin{figure}
    \centering
    \includegraphics[width=0.45\linewidth]{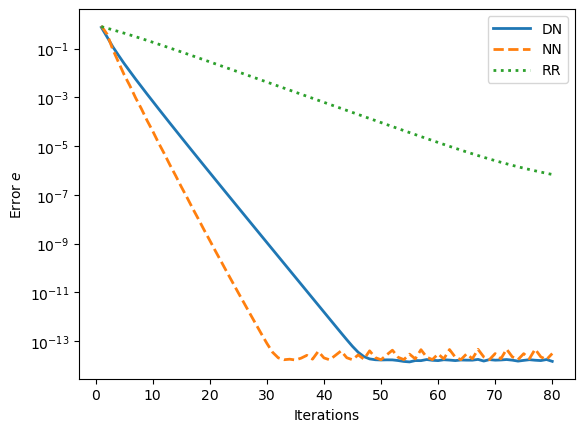}
    \includegraphics[width=0.45\linewidth]{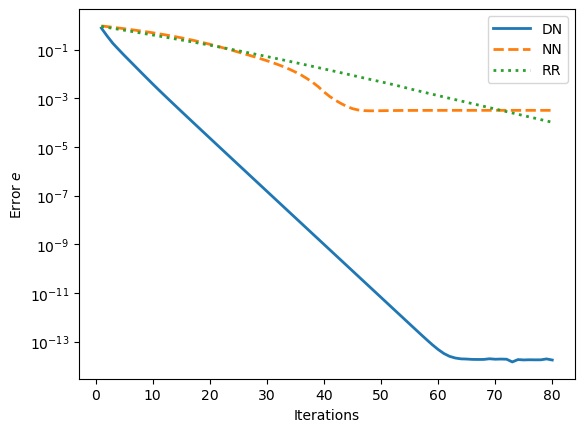}
    \caption{The error of the Dirichlet--Neumann method (DN) applied to~\cref{ex:ex1} (left) and~\cref{ex:ex2} (right) compared to the Neumann--Neumann (NN) and Robin--Robin (RR) methods.}
    \label{fig:err2}
\end{figure}%
    \bibliographystyle{siamplain}
\bibliography{ref}
\end{document}